\documentclass[a4paper,12pt]{amsart}

\usepackage{amsmath}
\usepackage{amssymb}
\usepackage{enumitem}
\usepackage{amsfonts}
\usepackage{graphicx}
\usepackage{cite}
\usepackage{mathtools}
\usepackage[colorlinks]{hyperref}
\newcommand{\R}{{\mathbb{R}^n}}
\title[Revised logarithmic Sobolev inequalities of fractional order]{Revised 
logarithmic Sobolev inequalities of fractional order}

\author[M. Chatzakou]{Marianna Chatzakou}
\address{
	Marianna Chatzakou:
	\endgraf
	Department of Mathematics: Analysis, Logic and Discrete Mathematics
	\endgraf
	Ghent University, Krijgslaan 281, Building S8, B 9000 Ghent
	\endgraf
	Belgium
	\endgraf
	{\it E-mail address} {\rm marianna.chatzakou@ugent.be}}

\author[M. Ruzhansky]{Michael Ruzhansky}
\address{
	Michael Ruzhansky:
	\endgraf
	Department of Mathematics: Analysis, Logic and Discrete Mathematics
	\endgraf
	Ghent University, Krijgslaan 281, Building S8, B 9000 Ghent
	\endgraf
	Belgium
	\endgraf
	and
	\endgraf
	School of Mathematical Sciences
		\endgraf Queen Mary University of London 
			\endgraf
		United Kingdom
			\endgraf
	{\it E-mail address} {\rm michael.ruzhansky@ugent.be}}

\thanks{The authors are supported by the FWO Odysseus 1 grant G.0H94.18N: Analysis and Partial Differential Equations and by the Methusalem programme of the Ghent University Special Research Fund (BOF) (Grant number 01M01021). M. Chatzakou is a postdoctoral fellow of the Research Foundation – Flanders (FWO)  under the postdoctoral grant No 12B1223N. M. Ruzhansky is also supported by EPSRC grant EP/R003025/2.}

\newtheoremstyle{theorem}
{10pt}          
{10pt}  
{\sl}  
{\parindent}     
{\bf}  
{. }    
{ }    
{}     
\theoremstyle{theorem}

\numberwithin{equation}{section}
\theoremstyle{plain}
\newtheorem{thm}{Theorem}[section]

\theoremstyle{definition}

\newtheorem{rem}[thm]{Remark}

\newtheoremstyle{defi}
{10pt}          
{10pt}  
{\rm}  
{\parindent}     
{\bf}  
{. }    
{ }    
{}     
\theoremstyle{defi}

\begin{document}

 	\begin{abstract}
In this short note we prove the logarithmic Sobolev inequality with derivatives of fractional order on $\R$ with an explicit expression for the constant. Namely, we show that for all $0<s<\frac{n}{2}$  and $a>0$ we have the inequality 
\[
 \int_{\R}|f(x)|^2 \log \left( \frac{|f(x)|^2}{\|f\|^{2}_{L^2(\R)}}\right)\,dx+\frac{n}{s}(1+\log a)\|f\|_{L^2(\R)}^{2}\leq C(n,s,a)\|(-\Delta)^{s/2}f\|^{2}_{L^2(\R)}
\]
with an explicit $C(n,s,a)$ depending on $a$, the order $s$, and the dimension $n$, and investigate the behaviour of $C(n,s,a)$ for large $n$. Notably, for large $n$ and when $s=1$, the constant $C(n,1,a)$ is asymptotically the same as the sharp constant of Lieb and Loss that was computed in \cite{LL01}. Moreover, we prove a similar type inequality for functions  $f \in L^q(\R)\cap W^{1,p}(\R)$ whenever $1<p<n$ and $p<q\leq \frac{p(n-1)}{n-p}$.
	\end{abstract}
	\maketitle

\section{Introduction}

The main purpose of this note is to extend the logarithmic Sobolev inequality by Lieb and Loss, see \cite{LL01}, which states that for $a>0$ we have
\begin{equation}
    \label{lsineq}
    \int_\R |f(x)|^2 \log\left(\frac{|f(x)|^2}{\|f\|_{L^2(\R)}^2} \right)\,dx+n(1+\log a)\|f\|_{L^2(\R)}^2\leq \frac{a^2}{\pi}\int_\R |\nabla f(x)|^2\,dx\,,
\end{equation}
where $dx$ stands for the Lebesgue measure on $\mathbb{R}^n$. The inequality \eqref{lsineq} is satisfied as an equality if and only if the function $f$ is a multiple of the function $\exp\left( -\frac{\pi |x|^2}{2 a^2}\right)$.

Particularly in Theorem \ref{main.thm}  we extend the logarithmic Sobolev inequality \eqref{lsineq} for $f \in W^{s,2}(\R)$  by considering derivatives of order $0<s<\frac{n}{2}$, and in Theorem \ref{THMpq} by considering functions  $f \in L^q(\R) \cap W^{1,p}(\R)$, where $1<p<n$ and that $p<q\leq \frac{p(n-1)}{n-p}$.

   At this point let us mention that, in the past other authors have attempted to extend the sharp logarithmic Sobolev inequality \eqref{lsineq} of Lieb and Loss for fractional order derivatives. To be precise, in \cite[Theorem 2.1]{CT05} the authors claim to have established the sharp version of the logarithmic Sobolev inequality on $\R$ for higher order derivatives. However, in their proof they estimate a term that appears after an application of Young’s inequality, using the Hausdorff–Young inequality, and it seems that the choices of the indices in the appearing norms cannot satisfy the assumptions on both inequalities simultaneously.

The logarithmic Sobolev inequality \eqref{lsineq} is related to the original logarithmic Sobolev inequality by Gross \cite{Gro75}
\begin{equation}
\label{or.ls}
 \int_\R |f(x)|^2\log\left(\frac{|f(x)|^2}{\|f\|_{L^2(d\mu)}^2} \right)\,d\mu \leq \frac{1}{\pi}\int_\R |\nabla f(x)|^2\,d\mu\,,   
\end{equation}
where $d\mu=e^{-\pi|x|^2}\,dx$ is the Gaussian measure, and the $L^2$ norm here is considered accordingly in the space $L^2(\R,d\mu)$. Logarithmic Sobolev inequalities and their Gross counterparts have also been studied in the general setting of Lie groups in \cite{CKR21}, to which we can also refer for a more detailed review of the literature.

 Let us now introduce the spaces of functions that are meaningful to our considerations:  for $\widehat{f}$ being the Fourier transform of $f \in L^1(\mathbb{R}^n)$, we define  
\[
\widehat{f}(\xi)=:\int_{\mathbb{R}^n}e^{-2\pi i x\xi}f(x)\,dx\,.
\]

The function $-\Delta f$ has a unique representation $\int_{\mathbb{R}^n}\widehat{-\Delta f}(\xi)e^{2\pi i x \cdot \xi}\,d\xi$, and we can write 
\[
\widehat{-\Delta f}(\xi)=(4\pi^2 |\xi|^2)\widehat{f}(\xi)\,.
\]
By the functional calculus, the operator $(-\Delta)^{s/2}$ can then be defined on the Fourier side as multiplication by $(2\pi |\xi|)^s$, i.e., we have 
\[
\widehat{(-\Delta)^{s/2}f}(\xi)=(2\pi|\xi|)^s \widehat{f}(\xi)\,.
\]
For $s>0$, the Sobolev space $W^{s,2}(\R)$ is the subspace of functions in $L^2(\R)$ for which 
\[
\|f\|_{W^{s,2}(\R)}^2:=\int_\R |\widehat{f}(\xi)|^2(4\pi^2|\xi|^2)^s\,d\xi< \infty\,,
\]
or, equivalently,  
\[
\|f\|_{W^{s,2}(\R)}=\|(-\Delta)^{s/2}f\|_{L^2(\R)}<\infty\,.
\]

More generally, for $p \geq 1$, the Sobolev space $W^{s,p}(\R)$ is defined as the subspace of $L^p(\R)$ for which
\[
\|f\|_{W^{s,p}(\R)}=\|(-\Delta)^{s/2}f\|_{L^p(\R)}<\infty\,.
\]
As usual, we will assume that $f \neq 0$.

\par Let us now present the two main results of the current note. 
\begin{thm}\label{main.thm}
Let $f\in W^{s,2}(\mathbb{R}^n)$, where $0<s<\frac{n}{2}$, and let $a>0$ be any real number. Then, we have the following inequality:
\begin{equation}
    \label{LS}
    \begin{split}
         \int_{\R}|f(x)|^2 \log \left( \frac{|f(x)|^2}{\|f\|^{2}_{L^2(\R)}}\right)\,dx & +\frac{n}{s}(1+\log a)\|f\|_{L^2(\R)}^{2}\\
         & \leq \frac{nea^2}{2s}C(n,s)\|(-\Delta)^{s/2}f\|_{L^2(\R)}^{2}\,,
    \end{split}
\end{equation}
where 
\begin{equation}
    \label{ConstantCsn}
    C(n,s)=\frac{\Gamma\left( \frac{n-2s}{2}\right)}{2^{2s}\pi^s\Gamma\left( \frac{n+2s}{2}\right)}\left( \frac{\Gamma(n)}{\Gamma\left(\frac{n}{2}\right)}\right)^{\frac{2s}{n}}\,.   
\end{equation}
\end{thm}
To give an idea of the asymptotic behavior of the constant on the right hand side, in Remark \ref{REM:infty}, using Stirling's formula, we show that for large $n\gg 1$, estimate \eqref{LS} can be written as 
\begin{multline}
   \int_{\R}|f(x)|^2 \log  \left( \frac{|f(x)|^2}{\|f\|^{2}_{L^2(\R)}}\right)\,dx+\frac{n}{s}(1+\log a)\|f\|_{L^2(\R)}^{2}\\
  \lesssim \frac{2^{s-1}a^2}{s\pi^s}e^{1-s}n^{1-s}\|(-\Delta)^{s/2}f\|_{L^2(\R)}^{2}\,.
\end{multline}
In particular, for $s=1$, asymptotically we have the same constant on the right hand side as the optimal constant $\frac{a^2}{\pi}$ in the inequality \eqref{lsineq} of Lieb and Loss.

\par For the proof of Theorem \ref{main.thm} we are going to use the classical Sobolev inequality in the Euclidean setting. Precisely, the sharp version of the Sobolev inequality for fractional orders was proved in \cite{CT04} and states that for $0<s< \frac{n}{2}$ and $q=\frac{2n}{n-2s}$ we have
\begin{equation}
    \label{Sob.sharp}
    \|u\|_{L^q(\R)}^{2}\leq C(n,s)\|u\|_{W^{s,2}(\R)}^{2}\,,
\end{equation}
where $u \in W^{s,2}(\mathbb{R}^n)$, and the sharp constant $C(n,s)$ is given by \eqref{ConstantCsn}. Inequality \eqref{Sob.sharp} for integer orders was proved in \cite{Swa92}. 

Inequality \eqref{Sob.sharp} holds true as an equality if and only if $u$ is a multiple of the function $(c^2+(x-d)^2)^{-(n-2s)/2}$, where $c \in \mathbb{R} \setminus \{0\}$ and $d \in \mathbb{R}$.
The best constant $C(n,s)$ was given earlier in the special cases where $s=1$, $s=1/2$ and $s=2$, in \cite{Au98,Au76,Ta76}, \cite{LL01} and \cite{DHM00,Vo93}, respectively.
Our second result is a variation of the previous result in the case where $f \in L^q(\R) \cap W^{1,p}(\R)$ and reads as follows:
\begin{thm}
    \label{THMpq}
   Assume that $1<p<n$ and that $p<q\leq \frac{p(n-1)}{n-p}$. Let $a>0$ be any real number. Then for $f \in L^q(\R) \cap W^{1,p}(\R)$, we have 
       
        \begin{multline} \label{pq}        
        \int_{\R} |f(x)|^q \log \left(\frac{|f(x)|^q}{\|f\|^{q}_{L^q(\R)}} \right)\,dx+(1+\log a)\frac{p(q-1)}{q-p}\|f\|_{L^q(\R)}^{q} \\
        \leq a\frac{p(q-1)}{q-p}\mathfrak{S}(n,p,q)\|\nabla f\|_{L^p{(\R)}}^{\theta}\|f\|^{1-\theta}_{L^q(\R)}\,,
         \end{multline}
         where \begin{equation}
    \label{theta}
    \theta=\frac{(q-p)n}{(q-1)(np-(n-p)q)}\,,
\end{equation}
and 
\begin{equation}
    \label{Snpq}
    \mathfrak{S}(n,p,q)=\left(\frac{q-p}{p\sqrt{\pi}} \right)^\theta\left(\frac{pq}{n(q-p)} \right)^{\frac{\theta}{p}}\left(\frac{\delta}{pq} \right)^{\frac{1}{r}}\left(\frac{\Gamma\left( q\frac{p-1}{q-p}\right)\Gamma\left( \frac{n}{2}+1\right)}{\Gamma\left( \frac{p-1}{p}\frac{\delta}{q-p}\right)\Gamma\left(n \frac{p-1}{p}+1 \right)} \right)^{\frac{\theta}{n}}\,,
\end{equation}
where $r=p\frac{q-1}{p-1}$ and $\delta=np-q(n-p)>0$. 
\end{thm}
The proof of Theorem \ref{THMpq} is a slight modification of the arguments used in the proof of Theorem \ref{main.thm}. In this case we make use of the following inequality shown in \cite{DD03}:

For $1<p<n$ and $p<q \leq \frac{p(n-1)}{n-p}$ we have
\begin{equation}\label{Sob.sharppq}
    \|f\|_{L^r(\R)} \leq \mathfrak{S}(n,p,q)\|\nabla f \|_{L^p(\R)}^{\theta}\|f\|_{L^q(\R)}^{1-\theta}\,,
\end{equation}
where $r=p \frac{q-1}{p-1}$, while $\theta$ and $\mathfrak{S}(n,p,q)$ are given in \eqref{theta} and \eqref{Snpq}, respectively. 

Let us point out that inequality \eqref{Sob.sharppq} is sharp in the sense that it is satisfied as an equality for multiples of the function $(1+c|x-x_0|^{\frac{p}{p-1}})^{-\frac{p-1}{q-p}}$, where $c \in \mathbb{R} \setminus \{0\}$, and $x_0 \in \R$. 

\begin{rem}
    Let us also note that when $q=\frac{p(n-1)}{n-p}$, then we have $\theta=1$ and $r$ takes the value $r=\frac{np}{n-p}$. In this case, the inequality \eqref{Sob.sharppq} becomes the sharp Sobolev inequality \eqref{Sob.sharp} in the case where $s=1$, which was shown independently by Aubin \cite{Au98,Au76} and Talenti \cite{Ta76}.
\end{rem}
The authors thank Roberto Bramati and Aidyn Kassymov for discussions.

\section{Proofs}
In this section, we present the proofs of our results. As usual, we assume $f \neq 0$.
\begin{proof}[Proof of Theorem \ref{main.thm}]
The first term on the left-hand side of inequality \eqref{LS} can be rewritten as
\begin{eqnarray}
    \label{thm1.a1}
    \int_\R |f(x)|^2 \log \left(\frac{|f(x)|^2}{\|f\|_{L^2(\R)}^{2}} \right)\,dx & = & \frac{1}{\varepsilon} \int_\R |f(x)|^2 \log \left(\frac{|f(x)|^{2}}{\|f\|^{2}_{L^2(\R)}} \right)^{\varepsilon}\,dx\nonumber\\
    & = & \frac{\|f\|^{2}_{L^2(\R)}}{\varepsilon}\int_\R \frac{|f(x)|^2}{\|f\|_{L^2(\R)}^{2}}\log \left( \frac{|f(x)|^2}{\|f\|_{L^2(\R)}^{2}}\right)^{\varepsilon}\,dx\,,
\end{eqnarray}
for an $\varepsilon>0$ that will be specifed later. By Jensen's inequality for concave nonlinearities and with a probability measure $\frac{|f(x)|^2}{\|f\|^{2}_{L^2(\R)}}\,dx$ we have the that 
\begin{equation}
    \label{Jensen1}
      \int_\R \frac{|f(x)|^2}{\|f\|_{L^2(\R)}^{2}}\log \left( \frac{|f(x)|^2}{\|f\|_{L^2(\R)}^{2}}\right)^{\varepsilon}\,dx \leq \log\left(\int_{\R} \frac{|f(x)|^{2\varepsilon+2}}{\|f\|_{L^2(\R)}^{2\varepsilon+2}}\,dx \right)\,,
\end{equation}
so that a combination of \eqref{thm1.a1} and \eqref{Jensen1} gives 
\begin{eqnarray}
    \label{thm1.b}
  \int_\R |f(x)|^2 \log \left(\frac{|f(x)|^{2}}{\|f\|^{2}_{L^2(\R)}} \right)\,dx & \leq &   \frac{\|f\|^{2}_{L^2(\R)}}{\varepsilon} \log\left(\int_{\R} \frac{|f(x)|^{2\varepsilon+2}}{\|f\|_{L^2(\R)}^{2\varepsilon+2}}\,dx \right)\nonumber\\
  & = & \frac{(\varepsilon+1)\|f\|^{2}_{L^2(\R)}}{\varepsilon} \log\left( \frac{\|f\|^{2}_{L^{2\varepsilon+2}(\R)}}{\|f\|_{L^2(\R)}^{2}} \right)\,.
\end{eqnarray}
One can easily check that for any $b,x>0$, the following inequality holds true:
\begin{equation}
    \label{ax}
    \log x \leq bx-\log b-1\,.
\end{equation}
In the next computations we choose $\varepsilon$ so that $2\varepsilon+2=\frac{2n}{n-2s}$ for $0<s<\frac{n}{2}$, and use the inequality \eqref{ax} for $b=ea^2$.  This choice of $\varepsilon$ allows us to apply the Sobolev inequality as in \eqref{Sob.sharp}, and we have 
\begin{equation*}
\label{thm1.c}
\begin{split}
 \int_\R |f(x)|^2 \log & \left(\frac{|f(x)|^{2}}{\|f\|^{2}_{L^2(\R)}} \right)\,dx  \leq  \frac{(\varepsilon+1)\|f\|^{2}_{L^2(\R)}}{\varepsilon} \log\left( \frac{\|f\|^{2}_{L^{2\varepsilon+2}(\R)}}{\|f\|_{L^2(\R)}^{2}} \right)\nonumber\\
 & \leq  \frac{(\varepsilon+1)\|f\|^{2}_{L^2(\R)}}{\varepsilon} \left(b \frac{\|f\|^{2}_{L^{2\varepsilon+2}}}{\|f\|^{2}_{L^2(\R)}}-(\log b +1) \right)\nonumber\\
 & =  \frac{\varepsilon+1}{\varepsilon} \left(ea^2\|f\|^{2}_{L^{2\varepsilon+2}(\R)}- [ 1+\log(ea^2)]\|f\|_{L^2(\R)}^{2} \right)\nonumber\\
 & \leq  \frac{\frac{n}{n-2s}}{\frac{2s}{n-2s}} \left(ea^2 C(n,s)\|f\|^{2}_{W^{s,2}(\R)}-2(1+\log a )\|f\|_{L^2(\R)}^{2} \right)\nonumber\\
 & =  \frac{n}{2s} \left(ea^2 C(n,s)\|f\|^{2}_{W^{s,2}(\R)}-2(1+\log a  )\|f\|_{L^2(\R)}^{2} \right)\nonumber \\
  & =  \frac{nea^2}{2s} C(n,s)\|f\|^{2}_{W^{s,2}(\R)}-\frac{n}{s}(1+\log a  )\|f\|_{L^2(\R)}^{2}\,.
 \end{split}
\end{equation*}
Finally, the latter inequality can be rewritten as in \eqref{LS}, and the proof is complete.
\end{proof}
As mentioned in the introduction, the inequality \eqref{lsineq} when $s=1$ is optimal, and so the constant on the right-hand side cannot be improved. However, as we show in the next remark, by the asymptotic analysis of the constant $C(n,s)$ for any $s$, one can recover the sharp constant in \eqref{lsineq} for $s=1$ and for large $n$. In the following remark, we perform an asymptotic analysis of the constant, similar to that in \cite{CKR22}.

\begin{rem}[On the constant $C(n,s)$ in \eqref{LS} for large $n$]\label{REM:infty}
 To look at the behaviour of our inequality for large $n \gg 1$, let us recall the following asymptotic approximation:
\[
\Gamma(x+\alpha) \sim \Gamma(x)x^{\alpha}\,,\quad \alpha \in \mathbb{C}\,.
\]
So for large $n$ we have
\[
\frac{\Gamma\left(\frac{n-2s}{2} \right)}{\Gamma\left(\frac{n+2s}{2} \right)}=\frac{\Gamma(n/2-s)}{\Gamma(n/2+s)}\sim \frac{\Gamma(n/2)(n/2)^{-s}}{\Gamma(n/2)(n/2)^{s}}=(n/2)^{-2s}\,.
\]
On the other hand, using Stirling's formula
\[
\Gamma(x)\sim {\sqrt\frac{2\pi}{x}} \left(\frac{x}{e} \right)^x\,,
\]
 we can approximate:
\[
\left( \frac{\Gamma(n)}{\Gamma(n/2)}\right)^{2s/n}\sim \left( \frac{e^{-n}n^{n-1/2}}{e^{-n/2}(n/2)^{n/2-1/2}}\right)^{2s/n}= \left(e^{-n/2}n^{n/2}2^{(n-1)/2} \right)^{2s/n}=e^{-s}n^{s}2^{s-s/n}\,. 
\]
Then for large $n$ by \eqref{ConstantCsn} we have
\[
C(n,s) \sim \frac{1}{2^{2s}\pi^s}(n/2)^{-2s}(e^{-s}n^{s}2^{s-s/n})=\frac{2^{s-s/n}}{\pi^s}e^{-s}n^{-s}\,.
\]
Hence \eqref{LS} for large $n$ becomes
\begin{equation}
\begin{split}
   \int_{\R}|f(x)|^2 \log & \left( \frac{|f(x)|^2}{\|f\|^{2}_{L^2(\R)}}\right)\,dx+\frac{n}{s}(1+\log a)\|f\|_{L^2(\R)}^{2}\\
 & \lesssim \frac{nea^2}{2s}\left(\frac{2^{s-s/n}}{\pi^s}e^{-s}n^{-s}\right)\|(-\Delta)^{s/2}f\|_{L^2(\R)}^{2}\\
 & = \frac{2^{s-1-s/n}a^2}{s\pi^s}e^{1-s}n^{1-s}\|(-\Delta)^{s/2}f\|_{L^2(\R)}^{2}\,.
\end{split}
\end{equation}
Thus, for large $n\gg 1$, we obtain the estimate
\begin{multline}
   \int_{\R}|f(x)|^2 \log  \left( \frac{|f(x)|^2}{\|f\|^{2}_{L^2(\R)}}\right)\,dx+\frac{n}{s}(1+\log a)\|f\|_{L^2(\R)}^{2}\\
  \lesssim \frac{2^{s-1}a^2}{s\pi^s}e^{1-s}n^{1-s}\|(-\Delta)^{s/2}f\|_{L^2(\R)}^{2}\,.
\end{multline}
 In particular, for $s=1$ and for large $n$, we have
 \[
  \int_\R |f(x)|^2 \log\left(\frac{|f(x)|^2}{\|f\|_{L^2(\R)}^2} \right)\,dx+n(1+\log a)\|f\|_{L^2(\R)}^2\lesssim \frac{a^2}{\pi}\int_\R |\nabla f(x)|^2\,dx\,, 
 \]
which is asymptotically the inequality \eqref{lsineq} of Lieb and Loss.

    \end{rem}

\begin{proof}[Proof of Theorem \ref{THMpq}]
    We have
\begin{eqnarray}
    \label{thm1.a}
    \int_\R |f(x)|^q \log \left(\frac{|f(x)|^q}{\|f\|_{L^q(\R)}^{q}} \right)\,dx & = & \frac{1}{\varepsilon} \int_\R |f(x)|^q \log \left(\frac{|f(x)|^{q}}{\|f\|^{q}_{L^q(\R)}} \right)^{\varepsilon}\,dx\nonumber\\
    & = & \frac{\|f\|^{q}_{L^q(\R)}}{\varepsilon}\int_\R \frac{|f(x)|^q}{\|f\|_{L^q(\R)}^{q}}\log \left( \frac{|f(x)|^q}{\|f\|_{L^q(\R)}^{q}}\right)^{\varepsilon}\,dx\,.
\end{eqnarray}
By Jensen's inequality we have that 
\begin{equation}
    \label{Jensen}
      \int_\R \frac{|f(x)|^q}{\|f\|_{L^q(\R)}^{q}}\log \left( \frac{|f(x)|^q}{\|f\|_{L^q(\R)}^{q}}\right)^{\varepsilon}\,dx \leq \log\left(\int_{\R} \frac{|f(x)|^{q\varepsilon+q}}{\|f\|_{L^q(\R)}^{q\varepsilon+q}}\,dx \right)\,.
\end{equation}
Combining \eqref{thm1.a} and \eqref{Jensen} we get 
\begin{eqnarray}
    \label{thm1.b}
  \int_\R |f(x)|^q \log \left(\frac{|f(x)|^{q}}{\|f\|^{q}_{L^q(\R)}} \right)\,dx & \leq &   \frac{\|f\|^{q}_{L^q(\R)}}{\varepsilon} \log\left(\int_{\R} \frac{|f(x)|^{q\varepsilon+q}}{\|f\|_{L^q(\R)}^{q\varepsilon+q}}\,dx \right)\nonumber\\
  & = & \frac{(\varepsilon+1)\|f\|^{q}_{L^q(\R)}}{\varepsilon} \log\left( \frac{\|f\|^{q}_{L^{q\varepsilon+q}(\R)}}{\|f\|_{L^q(\R)}^{q}} \right)\,.
\end{eqnarray}
In the following computations we make use of the inequality $\log x \leq ax-\log a-1$ which holds true for any $a,x>0$, and we choose $\varepsilon$ to be given by $\varepsilon=\frac{q-p}{q(p-1)}$ so that $q\varepsilon+q=p\frac{q-1}{p-1}$. The latter choice allows us to apply the inequality as in \eqref{Sob.sharppq}, and we have 
\begin{equation*}
\label{thm1.c}
\begin{split}
 \int_\R |f(x)|^q  & \log  \left(  \frac{|f(x)|^{q}}{\|f\|^{q}_{L^q(\R)}} \right)\,dx \\
 & \leq  \frac{(\varepsilon+1)\|f\|^{q}_{L^q(\R)}}{\varepsilon} \log\left( \frac{\|f\|^{q}_{L^{q\varepsilon+q}(\R)}}{\|f\|_{L^q(\R)}^{q}} \right)\nonumber\\
 & \leq  \frac{(\varepsilon+1)\|f\|^{q}_{L^q(\R)}}{\varepsilon} \left(a \frac{\|f\|^{q}_{L^{q\varepsilon+q}}}{\|f\|^{q}_{L^2(\R)}}-(\log a +1) \right)\nonumber\\
 & =  \frac{\varepsilon+1}{\varepsilon} \left(a\|f\|^{q}_{L^{q\varepsilon+q}(\R)}-(\log a +1)\|f\|_{L^q(\R)}^{q} \right)\nonumber\\
  & \leq  \frac{p(q-1)}{q-p} \left(a \mathfrak{S}(n,p,q)\|f\|^{\theta}_{L^p(\R)}\|f\|^{1-\theta}_{L^q(\R)}-(\log a  +1)\|f\|_{L^q(\R)}^{q} \right)\,.
 \end{split}
\end{equation*}
Observe that the latter inequality can be rewritten as in \eqref{pq}. The proof of Theorem \ref{THMpq} is complete.
\end{proof}
\begin{rem}
    Let us note that inequality \eqref{LS}, even for $s=1$, cannot be derived from inequality \eqref{pq} since by the assumptions in Theorem \ref{THMpq} we must have $p<q$.
\end{rem}

\end{document}